\documentclass{compositio}
\usepackage{epsfig}
\usepackage{amsmath}
\usepackage{amssymb}
\usepackage{amsthm}
\usepackage{verbatim}
\usepackage{boxedminipage}
\usepackage[all, knot]{xy}
\xyoption{arc} 

\newtheorem{theorem}{Theorem}
\newtheorem{proposition}[theorem]{Proposition}
\newtheorem{corollary}[theorem]{Corollary}
\newtheorem{lemma}[theorem]{Lemma}

\newtheorem*{unnumberedtheorem}{Theorem}
\theoremstyle{definition}
\newtheorem{definition}[theorem]{Definition}
\theoremstyle{remark}
\newtheorem{remark}[theorem]{Remark}

\newcommand{\A}{\mathbb{A}}

\newcommand{\Z}{\mathbb{Z}}

\newcommand{\Q}{\mathbb{Q}}

\newcommand{\lambdacan}{\lambda^{\mathrm{can}}}
\newcommand{\Qbar}{\overline{\mathbb{Q}}}

\newcommand{\cL}{\mathcal{L}}
\newcommand{\cG}{\mathcal{G}}
\newcommand{\cF}{\mathcal{F}}
\newcommand{\fF}{\mathfrak{F}}
\newcommand{\fq}{\mathfrak{q}}
\newcommand{\fQ}{\mathfrak{Q}}
\newcommand{\fw}{\mathfrak{w}}
\newcommand{\fl}{\mathfrak{l}}
\newcommand{\Gal}{\mathrm{Gal}\,}
\newcommand{\Sym}{\mathrm{Sym}}
\newcommand{\Ad}{\mathrm{Ad}\,}
\newcommand{\Frob}{\mathrm{Frob}}

\newcommand{\Hom}{\mathrm{Hom}\,}

\newcommand{\Prim}{\mathit{Prim}}

\newcommand{\ad}{\mathrm{ad}}

\newcommand{\Sp}{\mathrm{Sp}}
\newcommand{\GL}{\mathrm{GL}}
\newcommand{\PSL}{\mathrm{PSL}}
\newcommand{\SL}{\mathrm{SL}}
\newcommand{\ssm}{{^{\mathrm{ss}}}}
\newcommand{\Fbar}{\overline{F}}

\newcommand{\Kbar}{\overline{K}}
\newcommand{\F}{\mathbb{F}}
\newcommand{\eps}{\epsilon}

\newcommand{\notdiv}{\!\!\not|\,}
\newcommand{\begstrat}{
\par\vspace{-0.3\baselineskip}\hfill\begin{boxedminipage}[t]{4.7in}
\hspace{-0.3in}\begin{minipage}{4.9in}
\begin{description}
}

\newcommand{\ra}{\rightarrow}

\title[Potential automorphy of odd-dimensional Galois representations]{On the potential automorphy of certain odd-dimensional Galois representations}
\author{Thomas Barnet-Lamb}
\email{tbl@math.harvard.edu} 
\address{Department of Mathematics\\Harvard University\\Cambridge\\MA 02138\\USA} 

\begin{document}

% ÕMSC classification, keywords and grant acknowledgementsÕ
\subjclass{11R39 (primary), 11F23 (secondary)}
\keywords{Galois representation, potential automorphy, potential modularity, Dwork hypersurface}
\thanks{The author was partially supported by NSF grant DMS-0600716 and by a Jean E.~de Valpine Fellowship.} 

\begin{abstract}
In a previous paper, the potential automorphy of certain Galois representations to $\GL_n$ for $n$ even was established, following work of Harris, Shepherd-Barron and Taylor and using the lifting theorems of Clozel, Harris and Taylor. In this paper, we extend those results to $n=3$ and $n=5$, and conditionally to all other odd $n$. The key additional tools necessary are results which give the automorphy or potential automorphy of symmetric powers of elliptic curves, most notably those of Gelbert, Jacquet, Kim, Shahidi and Harris.
\end{abstract}
\maketitle

\section{Introduction}

In \cite{tbl-potmod}, it is established that certain even-dimensional Galois representations become automorphic when one makes a suitably large totally-real field extension. This paper has two main aims. The first is to extend those results to three- and five-dimensional representations; in particular, to prove the following theorem, which refers to a constant $C(n,N)$ which will be defined in the proof:

\begin{theorem} \label{main-theorem-1}
Suppose that $F/F_0$ is a Galois extension of CM fields and that $n=3$ or $5$. Suppose $N\geq n+6$ an even integer, and suppose that $l>C(n,N)$ is a prime which is unramified in $F$ and $l\equiv 1 \mod N$. Let $v_q$ be a prime of $F$ above a rational prime $q\neq l$ such that $q\notdiv N$. Let $\cL$ be a finite set of primes of $F$ not containing primes above $lq$.

Suppose that we are given a representation
$$r:\Gal (\Fbar/F) \ra \GL_n(\Z_l)$$
with the following properties (we will write $\bar{r}$ for the semisimplification of the reduction of $r$):
\begin{enumerate}
\item $r$ ramifies only at finitely many primes.
\item $r^c \cong r^\vee\eps_l^{1-n}$, with Bellaiche-Chenevier sign\footnote{More concretely: we can think of the isomorphism $r^c \cong r^\vee\eps_l^{1-n}$ as giving us a pairing $\langle*,*\rangle$ on $(\Z_l)^n$ satisfying $\langle r(\sigma) v_1,r({}^c\sigma) v_2\rangle=\eps_l^{1-n}(\sigma)\langle v_1,v_2\rangle$ for each $\sigma\in \Gal (\Fbar/F)$ and $v_1,v_2\in(\Z_l)^n$. If $r$ is in addition assumed to be absolutely irreducible, this pairing will either be symmetric or antisymmetric---and whether it is symmetric or antisymmetric turns out to only depend on $r$. We define the \emph{sign} of $r$ to be +1 if the pairing is symmetric, -1 if it is antisymmetric. This is the appropriate generalization of an `odd' two-dimensional Galois representation over a totally-real field; just as even representations are rather badly behaved, so are representations with sign -1; we would not expect to get good results for such representations.} +1
\item For each prime $\fw|l$ of $F$, $r|_{\Gal(\Fbar_\fw/F_\fw)}$ is crystalline with Hodge-Tate numbers $\{0,1,\dots, n-1\}$
\item $r$ is unramified at all the primes of $\cL$
\item $(\det \bar{r})^2\cong \eps_l^{n(1-n)}$ mod $l$
\item Let $r'$ denote the extension of $r$ to a continuous homomorphism $\Gal(\Fbar/F^+)\ra\cG_n(\Qbar_l)$ as described in section 1 of \cite{cht}; then $\bar{r}'(\Gal(\Fbar/F(\zeta_l))$ is `big' in the sense of `big image'. 
\item $\Fbar^{\ker\ad \bar{r}}$ does not contain $F(\zeta_l)$
\item $\bar{r}$ satisfies, for each prime $\fw|l$ of $F$:
$$\bar{r}|_{I_{F_w}} \cong 1 \oplus \eps_l^{-1} \oplus \dots \oplus  \eps_l^{1-n} $$
\item $r|_{\Gal(\Fbar_{v_q}/F_{v_q})}\ssm $ and $\bar{r}|_{\Gal(\Fbar_{v_q}/F_{v_q})}$ are unramified, with $r|_{\Gal(\Fbar_{v_q}/F_{v_q})} \ssm$ having Frobenius eigenvalues $1, (\#k(v_q)), \dots, (\#k(v_q))^{n-1}$
\end{enumerate}

Then there is a CM field $F'$ which is Galois over $F_0$ and linearly independent from  $\Fbar^{\ker\ad \bar{r}}$ over $F$. Moreover, all primes of $\cL$ and all primes of $F$ above $l$ are unramified in $F'$. Finally, there is a prime $w_q$ of $F'$ over $v_q$ such that $r|_{\Gal(\Fbar/F')}$ is automorphic of weight 0 and type $\{\Sp_n(1)\}_{\{w_q\}}$.
\end{theorem}

The second aim is to make, conditional on some conjectures of Michael Harris and co-workers which are expected to become theorems by 2010, a further extension to all remaining odd $n$, using work of Harris \cite{harris:pao} which establishes the potential automorphy of all odd-dimensional symmetric powers of elliptic curves subject to these `expected theorems'. To properly state these expected theorems and set them in the appropriate context takes most of the first section of \cite{harris:pao}, so I will not give a restatement of them here, but simply remark that they can be found as Expected Theorems 1.2, 1.4 and 1.7 there. For the remainder of this article, we shall refer to these statements as `the expected theorems of \cite{harris:pao}'.

We then have:
\begin{theorem} \label{main-theorem-2}
Suppose that we admit the expected theorems of \cite{harris:pao}, and let $F/F_0$ be a Galois extension of CM fields and $n$ an odd integer. Suppose further $N\geq n+6$ an even integer, that $l>C(n,N)$ is a prime which is unramified in $F$, and that $l\equiv 1 \mod N$. Let $\cL$ be a finite set of primes of $F$ not containing primes above $lq$.

Suppose that we are given a representation
$$r:\Gal (\Fbar/F) \ra \GL_n(\Z_l)$$
with the following properties:
\begin{enumerate}
\item $r$ ramifies only at finitely many primes.
\item $r^c \cong r^\vee\eps_l^{1-n}$, with Bellaiche-Chenevier sign +1
\item For each prime $\fw|l$ of $F$, $r|_{\Gal(\Fbar_\fw/F_\fw)}$ is crystalline with Hodge-Tate numbers $\{0,1,\dots, n-1\}$
\item $r$ is unramified at all the primes of $\cL$
\item $(\det \bar{r})^2\cong \eps_l^{n(1-n)}$ mod $l$
\item Let $r'$ denote the extension of $r$ to a continuous homomorphism $\Gal(\Fbar/F^+)\ra\cG_n(\Qbar_l)$ as described in section 1 of \cite{cht}; then $\bar{r}'(\Gal(\Fbar/F(\zeta_l))$ is `big' in the sense of `big image'. 
\item $\Fbar^{\ker\ad \bar{r}}$ does not contain $F(\zeta_l)$
\item $\bar{r}$ satisfies, for each prime $\fw|l$ of $F$:
$$\bar{r}|_{I_{F_w}} \cong 1 \oplus \eps_l^{-1} \oplus \dots \oplus  \eps_l^{1-n}$$
\end{enumerate}

Then there is a CM field $F'$ which is Galois over $F_0$ and linearly independent from  $\Fbar^{\ker\ad \bar{r}}$ over $F$. Moreover, all primes of $\cL$ and all primes of $F$ above $l$ are unramified in $F'$. Finally, $r|_{\Gal(\Fbar/F')}$ is automorphic of weight 0.
\end{theorem}

It is worth also remarking that, it is possible to use the freedom to vary $N$ to deduce corollaries where the congruence condition on $l$ is removed, replaced with a condition that $l$ avoid a certain set of primes of Dirichlet density zero. See Appendix \ref{anal-appendix} for the details. While these are strict weakenings of the theorems above, they may be useful for applications.

\begin{corollary} \label{cor-main-theorem-1}
Suppose that  $n$ is 3 or 5. There is a set $\Lambda$ of rational primes whose complement has Dirichlet density 0 and a function $N:\Lambda\ra\Z$ with the following property. Suppose that $F/F_0$ is a Galois extension of CM fields and $l\in \Lambda$ is a prime which is unramified in $F$. Let $v_q$ be a prime of $F$ above a rational prime $q\neq l$, and $q\notdiv N(l)$. Let $\cL$ be a finite set of primes of $F$ not containing primes above $lq$.

Finally, suppose that we are given a representation $r:\Gal (\Fbar/F) \ra \GL_n(\Z_l)$ enjoying the properties (1--9) of Theorem \ref{main-theorem-1}, and such that property (6) is `robust' in the sense that it remains true when $\bar{r}$ is restricted to any subgroup of $\Gal(\Fbar/F)$ with cyclic quotient.

Then there is a totally real field $F'$ which is Galois over $F_0$ and linearly independent from  $\Fbar^{\ker\ad \bar{r}}$ over $F$. Moreover, all primes of $\cL$ and all primes of $F$ above $l$ are unramified in $F'$. Finally, there is a prime $w_q$ of $F'$ over $v_q$ such that $r|_{\Gal(\Fbar/F')}$ is automorphic of weight 0 and type $\{\Sp_n(1)\}_{\{w_q\}}$.
\end{corollary}

\begin{corollary} \label{cor-main-theorem-2}
Suppose that we admit the expected theorems of \cite{harris:pao}, and let $n$ be an integer. There is a set $\Lambda$ of rational primes whose complement has Dirichelet density 0 and a function $N:\Lambda\ra\Z$ with the following property. Suppose that $F/F_0$ is a Galois extension of CM fields and $l\in \Lambda$ is a prime which is unramified in $F$. Let $\cL$ be a finite set of primes of $F$ not containing primes above $l$.

Finally, suppose that we are given a representation $r:\Gal (\Fbar/F) \ra \GL_n(\Z_l)$ enjoying the properties (1--8) of Theorem \ref{main-theorem-2}, and such that property (6) is `robust' in the sense that it remains true when $\bar{r}$ is restricted to any subgroup of $\Gal(\Fbar/F)$ with cyclic quotient.

Then there is a totally real field $F'$ which is Galois over $F_0$ and linearly independent from  $\Fbar^{\ker\ad \bar{r}}$ over $F$. Moreover, all primes of $\cL$ and all primes of $F$ above $l$ are unramified in $F'$. Finally $r|_{\Gal(\Fbar/F')}$ is automorphic of weight 0.
\end{corollary}

\medskip
I close the introduction with a simplified overview of the methods necessary to prove the two theorems above, with particular attention given to what is novel in the proofs. The basic structure of all potential automorphy proofs follows the seminal work of Taylor \cite{taylor2003rcf}. Two main ingredients are necessary. The first is a good supply of representations with certain properties, most notably that they are known \emph{a priori} to be automorphic. The second is a family $\fF$ of varieties whose cohomology is `very flexible'. By \emph{very flexible}, I mean that given specified $l$-adic and $l'$-adic representations $r$ and $r'$, and subject to certain conditions, we can find (over some suitably large totally real field) an element $V$ of the family whose mod $l$ cohomology looks like the residual representation of $r$, and whose mod $l'$ cohomology looks like the residual representation of $r'$. 

These two ingredients are then applied to prove the theorem as follows. We then apply the very flexible cohomology property taking $r$ to be the Galois representation which we would like to show is potentially modular and $r'$ to be one of our good supply of representations already known to be modular. Over the field of definition of this variety (which may well be a very large  extension of the field we started with) $r$ and the cohomology of $V$ (resp $r'$ and the cohomology of $V$) agree mod $l$ (resp $l'$). We can then apply a modularity lifting theorem twice; once, to deduce that the cohomology of $V$ is modular (since it agrees mod $l'$ with $r'$), then again to deduce that $r$ is modular (since it agrees mod $l$ with the cohomology of $V$). 

In practice, the actual argument involves a few more steps, since one has to accommodate various conditions that the lifting theorems have, most notably conditions that the representations we work with must be Steinberg at certain places. (For an overview explaining slightly more of the details, I refer the reader to the introduction to \cite{tbl-potmod}.) But this is more-or-less the flavor of the argument. 

The reason that the earlier paper had to restrict itself to even-dimensional representations was not a restriction in the part of the proof constructing the the family with very flexible cohomology, but rather in the supply of `good' Galois representations $r'$ already known to be modular. In particular, \cite{tbl-potmod} followed the paper \cite{hsbt} in taking these $r'$ to be induced from a character of a CM extension of $\Q$.  Such an $r'$ must necessarily be even dimensional, and so the results are automatically restricted to even dimensional representations.

The principal new idea in the present article is to use a different source of `already automorphic' Galois representations; in particular, we will use the symmetric powers of the cohomology of a suitably chosen elliptic curve. For $n=3$ and $n=5$, we will be able to take the symmetric square and fourth power, which are known unconditionally to be automorphic, and cuspidal in the appropriate cases, through work of Gelbert-Jacquet and Kim-Shahidi respectively \cite{gelbart1978rba, kim-sym4, kim2002csp}. For larger odd $n$, we use the result of \cite{harris:pao}, which as has already been mentioned proves, subject to the expected theorems, that all remaining symmetric powers are potentially automorphic. (Potential automorphy, rather than true automorphy, is good enough for our purposes.)

\medskip
The organization of the remainder of this paper is as follows. In section \ref{sec-dwork} we explain the very simple modifications necessary to the arguments of \cite{tbl-potmod} to extend the construction of motives with very flexible cohomology to the odd dimensional case. In section \ref{sec-review-powers} we briefly review the literature (\cite{gelbart1978rba, kim-sym4, kim2002csp}) on symmetric powers of 2 dimensional Galois representations and the results of Harris in \cite{harris:pao} and use these results to construct the supply of `good' representations which we need. Finally, in section \ref{sec-proof-main-thm} we put these pieces together to get a proof of the main theorem.

\medskip
{\bf Acknowledgements:} I would like to thank my advisor, Richard Taylor, for immeasurable assistance in all aspects of the preparation of this paper.  I would also like to thank Toby Gee for encouraging me to attempt to tackle the case of $n>5$ conditionally on the book project. Finally, I would like to thank both Barry Mazur, and the anonymous referee, for several helpful suggestions.

\section{Extending the analysis of the Dwork family} \label{sec-dwork}

In this section we describe how the geometric arguments in section 2 of \cite{tbl-potmod} can be extended to cover the case of odd $n$. The changes necessary are very straightforward. We follow the first page of section 2 of \cite{tbl-potmod}, which sets up notation, completely unaltered. The first difference between the analysis here and in \cite{tbl-potmod} comes when we reach the point just before Proposition 4 where a choice of a certain even-dimensional piece of the cohomology $\Prim_{l,[v]}$ is made, corresponding to a choice of an element $[v]\in (\Z/N\Z)_0/\langle W\rangle$. At this point, we assume that we have an \emph{odd} number $n$, $n>1$; we then write $n=2k+1$ and we choose a different $v$, viz:
\begin{align*}
v=(0,0,\dots,0,1,&k+2,k+3,\dots, N/2-2,N/2,\\&N/2+1,\dots, N-k-3, N-k-2,N-2)
\end{align*}
where we include every number once, except that we \emph{omit} the ranges $2,3,\dots,k+1$ and $N-k-1,\dots,N-3$, together with the singletons $N/2-1$ and $N-1$, and where the number of 0s at the beginning is $n+1$, calculated to ensure that there are $N$ numbers in total. Note that these numbers add up to 0 mod $N$. Note also that the ranges `make sense' as long as $N> n+3$; for instance, if $n=3, N=8$, we take $v=(0,0,0,0,1,4,5,6)$.

We then have the following analogues of Proposition 3, Corollary 4, and Lemma 6 of \cite{tbl-potmod}: (As usual, we write $\Prim_{l,[v],t}$ for the stalk of $\Prim_{l,[v]}$ at $t$, \emph{etc}.)

\begin{proposition} \label{geom-lemma}
We have the following facts about the varieties $Y_t$ and the sheaves
$\Prim[l]_{[v]}$, $\Prim_{l,[v]}$ and $\cF^i[l]$. Let $F$ be a a
number field. Recall that we are assuming $l\equiv 1$ mod $N$ throughout. 
\begin{enumerate}
\item \label{geom-lemma-good-red}
If $t\in T^{(l)}_0(F)$ and $\fq$ is a place of $F$ such that $v_\fq(1-t^N)=0$, then $Y_t$ has good reduction at $\fq$.
\item \label{geom-lemma-dual}
Suppose $t\in T^{(l)}_0(F)$. The Galois representation 
$$\Prim_{l,[v],t}:\Gal(\Fbar/F) \ra GL_n(\Q_l)$$  
satisfies $\Prim_{l,[v],t}^c \cong \Prim_{l,[v],t}^\vee \eps_l^{2-N}$. Similarly $\Prim[l]_{[v],t}^c \cong \Prim[l]_{[v],t}^\vee \eps_l^{2-N}$. (Indeed these isomorphisms patch for different $t$ to give a sheaf isomorphism.)
\item \label{geom-lemma-ht}
The sheaf $\Prim_{l,[v]}$ has rank $n$. There is a tuple $\vec{h}=(h(\sigma))_{\sigma\in\Hom(F,\Qbar_l)}$, such that the  Hodge-Tate numbers of $\Prim_{l,[v]}$ at the embedding $\sigma$ are $\{h(\sigma),h(\sigma)+1,\dots,h(\sigma)+n-1\}$.

\item \label{geom-lemma-inertia}
Let $\vec{h}$ continue to denote the tuple defined in the previous part. Suppose $\fw|l$, and let $\sigma\in\Hom(F,\Qbar_l)$ denote the corresponding embedding. Then $\Prim_{l,[v],0}|_{I_\fw}\cong \eps_l^{-h(\sigma)}\oplus\eps^{-h(\sigma)-1}_l\oplus\dots\oplus\eps_l^{1-h(\sigma)-n}$, and $\Prim[l]_{[v],0}|_{I_\fw}\cong\eps_l^{-h(\sigma)}\oplus\eps^{-h(\sigma)-1}_l\oplus\dots\oplus\eps_l^{1-h(\sigma)-n}$
\item \label{geom-lemma-steinberg}
Let $\fq$ be a prime of $F$ above a rational prime which does not divide $N$. If $\lambda_\fq\in T^{(l)}_0(F_\fq)$ has $v_\fq(\lambda_\fq) < 0$, then $(\Prim_{l,[v],\lambda_\fq})\ssm$ is unramified, and $(\Prim_{l,[v],\lambda_\fq})\ssm(\Frob_\fq)$ has eigenvalues $$\{\alpha,\alpha\#k(\fq),\alpha(\#k(\fq))^2,\dots,\alpha(\#k(\fq))^{n-1}\}$$ for some $\alpha$.
\item \label{geom-lemma-unram-mod}
Let $\fq$ be a prime of $F$ above a rational prime which does not divide $N$. If $\lambda_\fq\in T^{(l)}_0(F_\fq)$ has $v_\fq(\lambda_\fq) < 0$ and $l|v_\fq(\lambda_\fq)$, then $(\Prim[l]_{[v],\lambda_\fq})$ is unramified (even without semisimplification).
\item The monodromy of $\Prim_{l,[v]}$ is Zariski dense in $\{A\in\GL_n|\det A=\pm 1\}$.
\end{enumerate} 
\end{proposition}
\begin{proof} The proof of the corresponding proposition of \cite{tbl-potmod} only depends on $v$ insofar as it requires $v$ to possess the following properties:
\begin{itemize}
\item $v$ satisfies point (1) of the equivalent conditions in Lemma 10.1 of \cite{katz2007ald}, viz that the value 0 occurs more than once and no other value does.
\item $-v$ is not a permutation of $v$
\item $v$ omits precisely $n$ congruence classes mod $N$
\end{itemize}
Since our new $v$ is readily seen to possess these properties too, we can carry over the proof unchanged.
\end{proof}

\begin{corollary} \label{monodromy-mod-prop} There is a constant $C(n,N)$ such that if $M$ is an integer divisible only by primes $p>C(n,N)$ and if $t\in T_0^{(M)}$ then the map
$$\pi_1(T_0^{(M)},t) \ra \GL_n(\Prim[M]_{[v],t})$$
surjects onto $\SL^\pm_n(\Prim[M]_{[v],t})$. (Here $SL^\pm_n(\Prim[M]_{[v],t})$ denotes the group of automorphisms of $\Prim[M]_{[v],t}$ with determinant $\pm1$.) (We may, and shall, additionally assume that $C(n,N)>n$.)
\end{corollary}
\begin{proof} The argument is identical to the proof of Corollary 4 of \cite{tbl-potmod} or Lemma 1.11 of \cite{hsbt}, deducing the result from part (7) of the previous proposition and from Theorem 7.5 and Lemma 8.4 of \cite{mvw} (or Theorem 5.1 of \cite{nori1987sgn}).
\end{proof}

\begin{lemma} \label{det-lemma} Define a character $G_{\Q(\mu_N)}\ra\Q_l^\times$:
$$\phi_l:=\Lambda_{v,W}\prod_i (\lambdacan_{G_{\Q(\mu_N)}}(\{\chi_i\},\{1\}))^{2} $$
(where $\Lambda_{v,W}$ is the Galois character defined in the main Theorem 5.3 of \cite{katz2007ald} and the $\chi_i$ are the maps $\mu_N\ra\mu_N$ naturally associated to the elements $v_i\in\Z/N/\Z$); we have that 
$$(\det \Prim_{l,[v],t=2})^2 = \phi_l^{2n} \eps_l^{n(1-n)}$$
\end{lemma}
\begin{proof} Exactly as for Lemma 7 of \cite{tbl-potmod}.
\end{proof}

We can use this result to define a notation which will be useful to us in the remainder of this paper. Looking at the Hodge-Tate number of either side of the equation above at a prime $\fl$ over $l$, and writing $\mathrm{HT}_\fl(\phi_l)$ for the Hodge-Tate number of $\phi_l$ at that place, we get
\begin{align*}
2\times(\vec{h}(\fl)+(\vec{h}(\fl)+1)+\dots+(\vec{h}(\fl)+n-1)) &= 2n\,\mathrm{HT}_\fl(\phi_l) + n(n-1) \\
(2\vec{h}(\fl)+n-1)n &= 2n\,\mathrm{HT}_\fl(\phi_l)+ n(n-1) \\
(2\vec{h}(\fl))n &= 2n\,\mathrm{HT}_\fl(\phi_l)
\end{align*}
and we deduce that $\mathrm{HT}_\fl(\phi_l)=\vec{h}(\fl)$. Thus we can use twisting by $\phi_l$ to shift the Hodge-Tate numbers of an arbitrary representation by $\vec{h}$. 

\begin{definition}
Given an $l$-adic Galois representation $r$, we will write $r(-\vec{h})$ for the twist of $r$ by the character $\phi_l$ introduced above, and $r(\vec{h})$ for the twist by the inverse. (Thus, for example, the Hodge-Tate numbers of $r(\vec{h})$ are those of $r$ shifted by \emph{minus} $\vec{h}$.)
\end{definition}

We can finally deduce an analogue of Proposition 7 of \cite{tbl-potmod}. (We have slightly changed the notation.)

\begin{proposition} \label{geom-prop} The family $Y_t$ and the piece of its cohomology corresponding to $\Prim_{l,[v],t}$ have the following property:

Suppose $K'/K$ is a Galois extension of CM fields, with totally real subfields $K'^+, K^+$, $n$ is a positive odd integer, $l_1, l_2 \dots l_r$ are distinct primes which are unramified in $K$, and that we are given residual representations
$$\bar{\rho}_i:\Gal (\Kbar/K) \ra \GL_n(\F_{l_i}).$$
Suppose further that we are given $\fq_1, \fq_2, \dots, \fq_s$, distinct primes of $K$ above rational primes $q_1,\dots,q_s$ respectively, and $\cL$ a set of primes of $F$ not including the $\fq_j$ or any primes above the $l_i$. Suppose that each $q_j$ satisfies $q_j\notdiv N$. Finally, suppose that the following conditions are satisfied for each $i$:
\begin{enumerate}
\item \label{geom-prop-cdx-l-cn}
$l_i>C(n, N)$

\item \label{geom-prop-cdx-l-mod}
$l_i\equiv 1 $ mod $N$

\item \label{geom-prop-cdx-unram-at-cl}
$\bar{\rho}_i$ is unramified at each prime of $\cL$ and the $l_k$, $k\neq i$.

\item \label{geom-prop-cdx-inertia}
For each prime $\fw$ above $l_i$, we have that 
$$\bar{\rho}_i | _{I_\fw} \cong 1\oplus \epsilon_{l_i}^{-1}\oplus\dots\oplus  \epsilon_{l_i}^{1-n}$$

\item \label{geom-prop-cdx-dual}
We have that there exists a polarization $\bar{\rho}_i^c\cong \bar{\rho}_i^\vee \eps_{l_i}^{1-n}$; given this, we can associate to $\bar{\rho}_i$ a sign in the sense of Bellaiche-Chenevier and we require that this sign is +1. Finally, we require that $(\det\bar{\rho_i})^2\cong \eps_{l_i}^{n(1-n)}$
\end{enumerate}

Then we can find a CM field $F/K$, linearly disjoint from $K'/K$, a finite-order character $\chi_i:\Gal(\Qbar/F) \ra \Q_{l_i}$ for each $i$, and a $t\in F$ such that, 
\begin{enumerate}
\item \label{geom-prop-conclude-unram-and-split}
All primes of $K$ above the $\{l_i\}_{i=1,\dots,r}$ and all the $\cL$ are unramified in $F$
\item \label{geom-prop-conclude-good-reduction}
For all $i$, $Y_t$ has good reduction at each prime above lying above $l_i$, and each prime above the primes of $\cL$.
\item \label{geom-prop-conclude-crys-ht}
For all $i$ and $\fw|l_i$, $\Prim_{\fw,t}(\vec{h})\otimes\chi_i$ is crystalline with Hodge-Tate numbers $\{0,1,\dots,n-1\}$.
\item \label{geom-prop-conclude-steinberg}
For each $\fQ$ above some $\fq_j$, we have that $(\Prim_{l_i,[v],t})^{ss}$ and $\chi_i$ are unramified at $\fQ$, with $(\Prim_{l_i,[v],t}^{ss}(\vec{h})\otimes\chi)(\Frob_{\fQ})$ having eigenvalues $\{1, \#k(\fQ), \#k(\fQ)^2,\dots,\#k(\fQ)^{n-1}\}$.
\item \label{geom-prop-conclude-agree-mod}
$\Prim[l_i]_{[v],t}(\vec{h})\otimes\bar{\chi}_i = \bar{\rho}_i$ for all $i$.
\end{enumerate}
\end{proposition}
\begin{proof} The proof is identical to that of Proposition 7 of \cite{tbl-potmod}, except that all concerns about the sign of the pairing determinant are eliminated, since for odd-dimensional representations one can turn a pairing with square determinant into one with non-square determinant, and vice-virca, by multiplying the whole pairing by a non-square.
\end{proof}

\section{Review of certain automorphy and potential automorphy theorems} \label{sec-review-powers}

As was mentioned in the introduction, the key difficulty in proving a potential automorphy theorem for odd-dimensional representations is finding a good starting point: a supply of of odd-dimensional representations known to be (at least potentially) modular. We will use two sources here. On the one hand we have the functoriality theorems for the symmetric square and fourth power of a two dimensional representation, as proved by Gelbert, Jacquet, Shahidi and Kim. (The symmetric square will be three dimensional and the fourth power will be five dimensional.) Let us review these theorems:

\begin{theorem}[Gelbert-Jacquet; part of Theorem (9.3) of \cite{gelbart1978rba}] Let $F$ be a number field, and $\pi$ be a unitary irreducible representation of $\GL_2(\mathbb{A}_F)$ which is automorphic cuspidal. Assume that for any character $\chi$ of $\A^\times/F^\times$, $\chi\neq1$, we have that $\pi$ and $\pi\otimes\chi$ are inequivalent. Then $\Sym^2\pi$ is automorphic cuspidal\footnote{$\Sym^2\pi$ may be defined by putting together local pieces---the local definition comes from local Langlands.}.
\end{theorem}

\begin{theorem}[Kim; Theorem B of \cite{kim-sym4}] Suppose $F$ is a number field, and $\pi$ is a cuspidal automorphic representation of $\GL_2(\A_F)$. Then $\Sym^4\pi$ is an automorphic representation\footnote{$\Sym^4\pi$ is defined by putting together local pieces---the local definition comes from local Langlands.} of $\GL_5(\A_F)$. If $\Sym^3(\pi)$ is cuspidal, then $\Sym^4(\pi)$ is either cuspidal or induced from cuspidal representations of $\GL_2(\A_F)$ and $\GL_3(\A_F)$.
\end{theorem}

\begin{theorem}[Shahidi-Kim; Theorem 3.3.7 of \cite{kim2002csp}] Suppose $F$ is a number field, and $\pi$ is a cuspidal representation of $\GL_2(\A_F)$. Then $\Sym^4\pi$ is a cuspidal automorphic representation of $\GL_5(\A_F)$, unless:
\begin{itemize}
\item $\pi$ is monomial. (This is equivalent to $\pi$ being of dihedral type.)
\item $\pi$ is not monomial and $\Sym^3\pi$ is not cuspidal; this occurs when there exists a nontrivial gr\"ossencharacter $\mu$ such that $\Ad(\pi)\sim \Ad(\pi)\otimes\mu$. (This is equivalent to $\pi$ being of tetrahedral type.)
\item $\Sym^3\pi$ is cuspidal, but there exists a nontrivial quadratic character $\eta$ such that $\Sym^3(\pi)\sim \Sym^3(\pi)\otimes\eta$. (This is equivalent to $\pi$ being of octahedral type.)
\end{itemize}
\end{theorem}

Since a representation corresponding to a Galois representation with distinct Hodge-Tate weights will never be dihedral, tetrahedral or octahedral, we may conclude:
\begin{corollary} \label{autocorr} Suppose $F$ is a CM or totally real field, and $\pi$ is a cuspidal automorphic representation of $\GL_2(\A_F)$ corresponding to an irreducible Galois representation $r$ with Hodge-Tate numbers $\{0,1\}$ (for example, the cohomology of an elliptic curve), and such that there exists some place $v_q$ of $F$ such that $\pi$ has type $\{Sp(1)\}_{\{v_q\}}$.

Then $\Sym^2 r$ and $\Sym^4 r$ are automorphic of type $\{Sp(1)\}_{\{v_q\}}$; that is, there are RAESDC representations, Steinberg at $v_q$, viz $\Sym^2\pi$ and $\Sym^4\pi$, to which the Galois representations $\Sym^2 r$ and $\Sym^4 r$ are attached by the procedure of Harris-Taylor.
\end{corollary}
\begin{proof} Given the theorems above, we merely check the trivialities that $\pi$ being regular, being algebraic, being essentially-self-dual, and being Steinberg at $v_q$, imply the same for $\Sym^2\pi$ and $\Sym^4\pi$.
\end{proof} \medskip

This is our first source of `starting points' for proving odd-dimensional cases of potential automorphy. The second source is potential automorphy theorems for odd-dimensional symmetric powers of elliptic curves, as proved by Michael Harris.

\begin{theorem}[Harris; Theorem 4.4 of \cite{harris:pao}] \label{ha:thm} Let $n$ be an odd positive integer. Assume the expected theorems of \cite{harris:pao}. Let $F^{*,+}/F^+$ be an extension of totally real fields, $L$ a quadratic imaginary field, and $\cL$ a finite set of places of $F^+$. There is a field $M$ which depends only on $L$ and $n$ such that whenever $l>C'(n)$ is a rational prime unramified in $F^+$ and split in $M$, and $E$ an elliptic curve over $F^+$ with good reduction at $l$ and the primes of $\cL$, then there is a finite totally real Galois extension $F'^{,+}/F^+$, linearly disjoint from $F^{*,+}/F^+$ and unramified at primes of $\cL$, and a RAESDC automorphic representation $\pi$ of $\GL_n(F'^{,+})$, whose associated Galois representation is isomorphic to $(\Sym^n H^1(E,\Z_l))|_{\Gal(\Qbar/F'^{,+})}$. In other words, $\Sym^n H^1(E,\Z_l)|_{\Gal(\Qbar/F'^{,+})}$ is automorphic.

(Here $C'(n)$ is a constant depending only on $n$.)
\end{theorem}

\begin{remark}
The field $L$ and the conditions on $l$ do not occur in the statement of Theorem 4.4 in \cite{harris:pao}, since they reflect ongoing assumptions throughout that paper. For the convenience of the reader, I will point out exactly where the assumptions on $l$ arise. In the `set up' in \S 2.3, we assume that $l$ splits in $L$ and a certain field $\Q(\chi)$ depending only on $L$, and assume that $l>2n+1$. Then, in applying Corollary 2.5, we additionally require that $l$ is larger than some $C(n)$ and that it split in a certain cyclotomic extension. Finally, in applying Lemma 3.2 we assume $l>4n-1$. These conditions can be combined, as we do above, by saying that $l$ exceeds some $C'(n)$ and splits in some $M$.
\end{remark}

\begin{remark} \label{remark-add-curlyl} The set $\cL$ of places doesn't appear in the statement of Theorem 4.4 of \cite{harris:pao}, but the field $F'^{,+}$ is chosen by an application of Theorem 2.1 of that paper; this theorem allows us to ensure that the field we choose does not ramify at some finite set of primes; the condition that $E$ have good reduction at the primes of $\cL$ means we can just add the primes of $\cL$ to that set.\footnote{We might also have to be careful in our choice of the character $\chi$, so that it is unramified at the primes of $\cL$, but this is also easily achieved.} Similarly, the field $F^{*,+}$ does not occur in the statement of Theorem 4.4, but Theorem 2.1 allows us to avoid any fixed field; see Harris' remarks immediately after the statement of Theorem 2.1.
\end{remark}

We are now in a position to address the main theorems.

\section{Arguments for the main theorems} \label{sec-proof-main-thm}

We will give a proof of both of our main theorems simultaneously; let us restate them in a convenient form to do so. (We have also relabelled the set $\cL$ of primes as $\cL'$; this will avoid confusion from a notation clash with the different $\cL$'s which arise in the statements of propositions and theorems which we will apply in the proof.)

\begin{unnumberedtheorem}[Restatement of Theorems \ref{main-theorem-1} and \ref{main-theorem-2}]
Suppose that $F/F_0$ is a Galois extension of CM fields and that $n$ is an odd integer. Suppose $N\geq n+6$ an even integer. Suppose that $l>C(n,N)$ is a prime which is unramified in $F$ and $l\equiv 1 \mod N$. Let $\cL'$ be a finite set of primes of $F$ not containing primes above $lq$.

Suppose that we are given a representation
$$r:\Gal (\Fbar/F) \ra \GL_n(\Z_l)$$
with the following properties: (we will write $\bar{r}$ for the semisimplification of the reduction of $r$)
\begin{enumerate}
\item $r$ ramifies only at finitely many primes.
\item $r^c \cong r^\vee\eps_l^{1-n}$, with sign +1
\item $r$ is unramified at all the primes of $\cL'$
\item $(\det \bar{r})^2\cong \eps_l^{n(1-n)}$ mod $l$
\item Let $r'$ denote the extension of $r$ to a continuous homomorphism $\Gal(\Fbar/F^+)\ra\cG_n(\Qbar_l)$ as described in section 1 of \cite{cht}; then $\bar{r}'(\Gal(\Fbar/F(\zeta_l))$ is `big' in the sense of `big image'. 
\item $\Fbar^{\ker\ad \bar{r}}$ does not contain $F(\zeta_l)$
\item For each prime $w|l$ of $F$, $r|_{\Gal(\Fbar_w/F_w)}$ is crystalline with Hodge-Tate numbers $\{0,1,\dots, n-1\}$. Also: 
$$\bar{r}|_{I_{F_w}} \cong 1 \oplus \eps_l^{-1} \oplus \dots \oplus  \eps_l^{1-n} $$
\end{enumerate}

Suppose further that \emph{one} of the following holds:
\begin{description}
\item[Case X] We have that $n=3$ or $n=5$. Moreover, there exists $v_q$, a prime of $F$ above a rational prime $q\neq l$ such that $q\notdiv N$, such that $(r|_{\Gal(\Fbar_{v_q}/F_{v_q})})\ssm $ and $\bar{r}|_{\Gal(\Fbar_{v_q}/F_{v_q})}$ are unramified and  $(r|_{\Gal(\Fbar_{v_q}/F_{v_q})}) \ssm$ has Frobenius eigenvalues $1, (\#k(v_q)), \dots, (\#k(v_q))^{n-1}$.
\item[Case Y] We admit the expected theorems of \cite{harris:pao}.
\end{description}

Then there is a CM field $F'$ which is Galois over $F_0$ and linearly independent from  $\Fbar^{\ker\ad \bar{r}}$ over $F$. Moreover, all primes of $\cL'$ and all primes of $F$ above $l$ are unramified in $F'$. Finally, there is a prime $w_q$ of $F'$ over $v_q$ such that $r|_{\Gal(\Fbar/F')}$ is automorphic of weight 0 and type $\{\Sp_n(1)\}_{\{w_q\}}$.
\end{unnumberedtheorem}

\begin{proof} \emph{Step A:} We begin by choosing a quadratic imaginary field $L$ linearly disjoint from $F$ over $\Q$. Let $M$ be the field given in Theorem \ref{ha:thm}. We then choose a prime $l'$ with the following properties.
\begin{itemize}
\item $l'$ is unramified in $F$. \hfill(A1)
\item $l'>C(2)$ (this is the constant $C(n)$ defined in Theorem 3.2 of \cite{hsbt}) \hfill(A2x)
\item $l'>C'(n)$ (the constant from Theorem \ref{ha:thm}) \hfill(A2y)
\item $l'>C(n,N)$ (the constant from Proposition \ref{geom-prop}) \hfill(A3)
\item $l'$ splits in $\Q(\zeta_N)$ and in $M$ \hfill(A4)
\item $r$ is unramified at $l'$\hfill(A5)
\item $l'\neq l$ and $l'\neq q$
\end{itemize}
which is clearly possible.

\smallskip\emph{Step B:} Choose an elliptic curve $E$ over $\Q$ with the following properties:
\begin{itemize}
\item $E$ has good ordinary reduction at $l'$, with $H^1(E\times\Qbar,\Z_{l'})$ semisimple (or in other words, tamely ramified) \hfill(B1)
\item $E$ has good ordinary reduction at $l$ and the primes of $\cL'$ \hfill(B2)
\item \emph{(If we are in case {\bf X})} $E$ has multiplicative reduction at $q$. \hfill(B3)
\item The Galois representation coming from the cohomology $H^1(E\times\Qbar,\Z_{l'})$ is surjective.\hfill(B4)
\end{itemize}
It is possible to do this using the form of Hilbert irreducibility with weak approximation (see \cite{ekedahl1988evh}). We can impose the condition over $q$ by insisting that the $j$ invariant satisfies $v_q(j)<0$; we can impose the condition at the prime $l'$ by taking the Serre-Tate canonical lift of an ordinary elliptic curve.

We will write $r_E$ for the $n$ dimensional Galois representation given by $\Sym^{n-1}H^1(E\times \Qbar,\Z_{l'})$.

\smallskip\smallskip\emph{Step C:} We now apply Proposition \ref{geom-prop} with $K=F$, $r=2$, $l_1=l$, $l_2=l'$, $K'=\Fbar^{\ker \ad \bar{r}\cap\ker \ad \bar{r}_E}$, $\bar{\rho_1}=r$ mod $l$, $\bar{\rho_2}=r_E|_{G_F}$ mod $l'$, and $\cL=\cL'$. In case {\bf X} we take $s=1$ and $\fq_1=v_q$; in case {\bf Y} we take $s=0$. Let us verify the conditions of this proposition in turn:
\begin{enumerate}
\item \emph{The $l_i$ are large enough.} For $l$, this is a hypothesis of the theorem. For $l'$, it is point A3 above.
\item \emph{The $l_i$ split in $\Q(\zeta_N)$.} Again, for $l$ this is a hypothesis of the theorem; for $l'$ it is point A4 above.
\item \emph{The $\bar{\rho}_i$ are unramified at the primes of $\cL'$ and the $l_i$.} This follows from (B2), hypothesis (vi), and point (A5). 
\item \emph{The $\bar{\rho}_i$ have the right restriction to inertia.} For $\bar{\rho_1}$ this is hypothesis (vii) of our theorem; for $\bar{\rho_2}$ it follows from point (B1) above.
\item \emph{Essentially self dual with correct determinant and sign.}
  For $\bar{\rho_1}$ this is hypotheses (i) and (ii) of the theorem being proved; for $\bar{\rho_2}$ it follows from the fact that $r_E$ is symplectic with multiplier $\eps_{l'}^{1-n}$ and $r_E=r_E^c$.
\end{enumerate}
This constructs a CM field $F_0/F$, characters $\chi_1:G_F\ra Q_{l}$ and $\chi_2:G_F\ra Q_{l'}$, and a $t\in F_0$ such that:
\begin{itemize}
\item All primes of $F$ above $l,l'$ and all the primes of $\cL'$ are unramified in the field $F_0$ \hfill(C1)
\item $Y_t$ has good reduction at all primes of $F_0$ above $l,l'$ and all the primes of $\cL'$\hfill(C2)
\item For all $\fw|ll'$, $\Prim_{\fw,t}(\vec{h})\otimes\chi_i$ is crystalline with Hodge-Tate numbers $\{0,1,\dots,n-1\}$\hfill(C3)
\item \emph{(If we are in case {\bf X})} For each $\fQ$ above $v_q$, we have that $(\Prim_{l_i,[v],t})^{ss}$ and $\chi_i$ are unramified at $\fQ$, with $(\Prim_{l_i,[v],t}^{ss}(\vec{h})\otimes\chi)(\Frob_{\fQ})$ having eigenvalues $\{1, \#k(\fQ), \#k(\fQ)^2,\dots,\#k(\fQ)^{n-1}\}$ \hfill(C4)
\item We have that $\Prim[l]_{[v],t}(\vec{h})\otimes\bar{\chi}_i = r|_{G_{F_0}}\text{ mod $l$}$\hfill(C5)
\item We have that $\Prim[l']_{[v],t}(\vec{h})\otimes\bar{\chi}_i = r_E|_{G_{F_0}}\text{ mod $l'$}$\hfill(C6)
\item $F_0$ is linearly disjoint from $\Fbar^{\ker \ad \bar{r}\cap\ker \ad \bar{r}_E}$ over $F$.\hfill(C7)
\end{itemize}

\smallskip\emph{Step D:} The objective of this step is to find a totally real Galois field extension $F_1^+/F^+_0$, with the following properties:
\begin{itemize}
\item $r_E|_{G_{F_1^+}}$ is automorphic. (In case {\bf X}, automorphic of type $\{\Sp_n(1)\}_{\{w|v_q\}}$.) \hfill(D1)
\item None of the primes of $\cL'$ ramify in $F_1$\hfill(D2)
\item $l$ does not ramify in $F_1$\hfill(D3)
\item $F_1^+$ is linearly disjoint from $(\Fbar^{\ker \ad \bar{r}\cap\ker \ad \bar{r}_E})^+$ over $F^+$\hfill(D4)
\end{itemize}
We will use slightly different arguments in case {\bf X} and case {\bf Y}.
\begin{description}
\item[Case X] We apply Theorem \ref{ha:thm}, taking $\cL=\cL'\cup\{l\}$, $F^{*,+}=(\Fbar^{\ker \ad \bar{r}\cap\ker \ad \bar{r}_E})^+$ and with the $l$ in the theorem being our $l'$. We observe that $l'$ is split in $M$ (by condition A4), and that $E$ has good reduction at $l,l'$ and the primes of $\cL'$ (by conditions B1, B2 above), that $l'$ is large enough (by point A2x), and finally that the primes of $\cL'$, $l$ and $l'$ do not ramify in $F_1^+$ by conditions C1, A1 above. Thus we meet the conditions of the theorem, which immediately gives us what we want.
\item[Case Y] We apply Theorem 3.2 of \cite{hsbt}, \emph{in the case n=2}, to the \emph{two-dimensional} Galois representation $H^1(E\times\Qbar,\Z_{l'})$. In fact, we will apply a modified version of the theorem with a collection $\cL$ of primes which may be chosen not to ramify in the field extension the theorem will produce, and where we are allowed to assume this extension is linearly disjoint from any other fixed extension $F^{*,+}$; the requisite arguments are just as in Remark \ref{remark-add-curlyl} above. We will take $\cL=\cL'\cup\{l\}$, $F^{*,+}=(\Fbar^{\ker \ad \bar{r}\cap\ker \ad \bar{r}_E})^+$, and the $l$ in the theorem is our $l'$. 

It is immediate that the representation $H^1(E\times\Qbar,\Z_{l'})$ is odd and that it has Hodge-Tate numbers $\{0,1\}$. It is also surjective (by condition B4) and Steinberg at primes above $q$ (by condition B3). Finally note that the primes of $\cL'$, $l$ and $l'$ do not ramify in $F_1^+$ by conditions C1, A1 above and $l'$ is large enough to apply the theorem by condition A2y. Whence we satisfy the conditions of the theorem, and we can construct a field $F_1^+/F^+_0$ with $H^1(E\times\Qbar,\Z_{l'})|_{G_{F_1^+}}$ automorphic of type $\{\Sp_n(1)\}_{\{w|v_q\}}$. Applying Corollary \ref{autocorr}, and the fact that $n=3$ or $n=5$, we deduce that $r_E|_{G_{F_1^+}}$ (=$\Sym^{n-1}H^1(E\times\Qbar,\Z_{l'})|_{G_{F_1^+}}$) is automorphic of type $\{\Sp_n(1)\}_{\{w|v_q\}}$.
\end{description}

\smallskip\emph{Step E:} Define $F_1=F_0F_1^+$, a CM field with totally real subfield $F_1^+$. We now apply a modularity lifting theorem to deduce that the Galois representation $(\Prim_{l',[v],t}(\vec{h})\otimes\chi_1)|_{G_{F_1}}$ is automorphic (in case {\bf X}, automorphic of type $\{\Sp_n(1)\}_{\{w|v_q\}}$.) In case {\bf X}, the theorem we apply is Theorem 5.2 of \cite{taylor2006asa}; in case {\bf Y} we apply the strengthening of that theorem made possible by admitting the expected theorems of \cite{harris:pao} and given as Theorem 1.7 of that paper, in which the Steinberg condition is removed. Let us check the conditions of these theorems (following the numbering in \cite{taylor2006asa}) in turn:
\begin{enumerate}
\item \emph{$\Prim_{l',[v],t}(\vec{h})\otimes\chi_1|_{G_{F_1}}$ is conjugate self dual.} This is immediate from Proposition \ref{geom-lemma}, part 2.
\item \emph{$\Prim_{l',[v],t}(\vec{h})\otimes\chi_1|_{G_{F_1}}$ is unramified a.e.} This is trivial.
\item \emph{$\Prim_{l',[v],t}(\vec{h})\otimes\chi_1|_{G_{F_1}}$ is crystalline at all places above $l'$.} This is from point (C3) above.
\item \emph{$\Prim_{l',[v],t}(\vec{h})\otimes\chi_1|_{G_{F_1}}$ has the right Hodge-Tate numbers all places above $l'$.}  This is true for the same reason as the previous point.
\item \emph{$\Prim_{l',[v],t}(\vec{h})\otimes\chi_1|_{G_{F_1}}$ is Steinberg at places above $v_q$. (This condition only present in case {\bf X}.)} This is point (C4) above.
\item \emph{$\Prim_{l',[v],t}(\vec{h})\otimes\chi_1|_{G_{F_1}}$ has big image.} By point (C6), it suffices to show that $\bar{r}_E|_{G_{F_1}}$ has big image; by points (C7) and (D4) it then suffices to show $\bar{r}_E$ has big image. For this we use the simplicity of $\PSL_2(\F_l)$ for $l>3$, Corollary 2.5.4 of \cite{cht}, and point (B4) above.
\item \emph{Let $M'=\ker\ad\, \Prim[l']_{[v],t}(\vec{h})\otimes\chi_1|_{G_{F_1}}$; then $\Fbar^{M'}$ does not contain $F(\zeta_l)$.} Same argument as previous point. 
\item \emph{$\Prim_{l',[v],t}(\vec{h})\otimes\chi_1|_{G_{F_1}}$ is residually automorphic.} (In case {\bf X} we must additionally have `automorphic of type $\{\Sp_n(1)\}_{\{w|v_q\}}$'.) By point (D1) above, we have that $r_E|_{G_{F_1^+}}$ is automorphic; by abelian base change, we therefore have that $r_E|_{G_{F_1}}$ is automorphic. By point (C6) above, this gives us what we need.
\end{enumerate}

\smallskip\emph{Step F:} We now apply a modularity lifting theorem to deduce that $r|_{\Gal(\bar{F}/F_1)}$ is automorphic (in case {\bf X}, automorphic of type $\{\Sp_n(1)\}_{\{w|v_q\}}$.) Again we use Theorem 5.2 of \cite{taylor2006asa} in case {\bf X} and Theorem 2.7 of \cite{harris:pao} otherwise. Let us check the conditions:
\begin{enumerate}
\item \emph{$r|_{G_{F_1}}$ is conjugate self dual.} This is condition (ii) of the theorem currently being proved.
\item \emph{$r|_{G_{F_1}}$ is unramified a.e.} This is condition (i) of the theorem currently being proved.
\item \emph{$r|_{G_{F_1}}$ is crystalline at all places above $l$.} This is by condition (vii) of the theorem currently being proved.
\item \emph{$r|_{G_{F_1}}$ has the right Hodge-Tate numbers all places above $l$.}  This is also by condition (vii) of the theorem currently being proved.
\item \emph{$r|_{G_{F_1}}$ is Steinberg at places above $v_q$. (This condition only present in case {\bf X}.)} This is one of the hypotheses in case {\bf X}.
\item \emph{$\bar{r}|_{G_{F_1}}$ has big image.} Condition (v) of the theorem currently being proved gives that this is true before restriction to $G_{F_1}$; then points (C7) and (D4) show it remains true after this restriction.
\item \emph{$\Fbar^{\ker\ad \,\bar{r}|_{G_{F_1}}}$ does not contain $F(\zeta_l)$.} This is condition (vi) of the theorem currently being proved.
\item \emph{$r|_{G_{F_1}}$ is residually automorphic. (In case {\bf X} we must additionally have `automorphic of type $\{\Sp_n(1)\}_{\{w|v_q\}}$'.)} By point (C5) above, $\Prim[l]_{[v],t}(\vec{h})\otimes\chi_1=\bar{r}$, so certainly $\Prim[l]_{[v],t}(\vec{h})|_{G_{F_1}}\otimes\chi_1=\bar{r}|_{G_{F_1}}$ and it suffices to show $\Prim_{l,[v],t}(\vec{h})|_{G_{F_1}}$ is automorphic. For this, note that we concluded in step E that $(\Prim_{l',[v],t}(\vec{h})\otimes\chi_1)|_{G_{F_1}}$ is automorphic. Additionally, we know that $Y_t$ has good reduction at $l,l'$, by point (C2) above. Thus we can deduce that $(\Prim_{l,[v],t}(\vec{h})\otimes\chi_1)|_{G_{F_1}}$ is also automorphic.
\end{enumerate}
This concludes the proof of the theorem, since we can take $F'=F_2$. (Note that the primes of $\cL'$ are unramified in $F_2$ by condition D2, and the primes above $l$ are unramified by condition D3.)
\end{proof}

\appendix
\section{Some simple analysis}\label{anal-appendix}
The purpose of this section is to do some very simple analysis to allow the deduction of Corollaries \ref{cor-main-theorem-1} and \ref{cor-main-theorem-2} from our main theorems. It is trivial that it suffices to prove the following statement:
\begin{proposition} Fix an integer $n$. Let $\Lambda$ be a set of rational primes such that for all even $N>n$ there exists a constant $C(N)$ such that all primes $l$ congruent to 1 mod $N$ and larger than $C(N)$ lie in $\Lambda$. Then $\Lambda$ has Dirichlet density 1.
\end{proposition}
\begin{proof} Let $\epsilon>0$ be given. Then we can find a finite list of even integers $N_1,\dots,N_k$ each of which is $>n$ such that the set of primes congruent to 1 mod at least one $N_i$ has Dirichlet density $>1-\eps$. (For instance we could take the $N_i$ to simply be twice an increasing list of consecutive primes above $n$; then the fact that $\prod(1-(1/p))$ diverges to 0 gives us what we want.)

Then, writing $D^+(S)$ (resp $D^-(S)$) for the upper (resp lower) Dirichlet density of a set of primes $S$, $D(S)$ for the Dirichlet density if it exists, $S_1$ for the set of primes congruent to 1 mod at least one $N_i$ and $S_0$ for the set of primes congruent to 1 mod at least one $N_i$ and larger than the maximum of the $C(N_i)$, we have
\begin{align*}
1 &\geq D^+(\Lambda)\geq D^-(\Lambda) \\
&=D^-(S_0)&&\text{(since $S_0\subset\Lambda$)}\\
&=D^-(S_1)&&\text{(finite sets do not affect density)}\\
&=D(S_1) &&\text{(since $S_1$ has a density)}\\
&> 1-\eps
\end{align*} 
and since $\eps$ was arbitrary, we have that 
$1 \geq D^+(\Lambda)\geq D^-(\Lambda)\geq1$
whence we are done.
\end{proof}

\nocite{cht, deligne1989hcm, hsbt, harris:pao, katz2007ald, katz1990esde, gelbart1978rba, taylor2006asa, kim-sym4, kim2002csp}

\bibliographystyle{alpha} 
\bibliography{tbl-bib}

\begin{thebibliography}{HSBT06}

\bibitem[BL08]{tbl-potmod}
T.~Barnet-Lamb.
\newblock {Potential automorphy for certain Galois representations to
  $GL(2n)$}.
\newblock {\em preprint}, 2008.

\bibitem[CHT05]{cht}
L.~Clozel, M.~Harris, and R.~Taylor.
\newblock {Automorphy for some $l$-adic lifts of automorphic mod $l$ Galois
  representations.}
\newblock {\em preprint}, 2005.

\bibitem[Eke90]{ekedahl1988evh}
T.~Ekedahl.
\newblock {\em {An effective version of Hilbert's irreducibility theorem, in
  ``S\'eminaire de Theorie des Nombres, Paris, 1988-1989, Prog. in Math.
  91''}}, pages 241--249.
\newblock Birkh\"auser, 1990.

\bibitem[GJ78]{gelbart1978rba}
S.~Gelbart and H.~Jacquet.
\newblock {A relation between automorphic representations of $\GL(2)$ and
  $\GL(3)$}.
\newblock {\em Ann. Sci. Ecole Norm. Sup.(4)}, 11(4):471--542, 1978.

\bibitem[Har07]{harris:pao}
M.~Harris.
\newblock {Potential automorphy of odd-dimensional symmetric powers of elliptic
  curves, and applications}.
\newblock {\em Algebra, Arithmetic and Geometry--Manin Festschrift, Progress in
  Mathematics, Birkh\"auser, to appear}, 2007.

\bibitem[HSBT06]{hsbt}
M.~Harris, N.~Shepherd-Barron, and R.~Taylor.
\newblock {A family of Calabi-Yau varieties and potential automorphy}.
\newblock {\em preprint}, 2006.

\bibitem[Kat90]{katz1990esde}
N.M. Katz.
\newblock {\em {Exponential Sums and Differential Equations, Ann. Math. Study
  124}}.
\newblock Princeton University Press, 1990.

\bibitem[Kat07]{katz2007ald}
N.M. Katz.
\newblock {Another look at the Dwork family}.
\newblock {\em Algebra, Arithmetic and Geometry--Manin Festschrift, Progress in
  Mathematics, Birkh\"auser, to appear}, 2007.

\bibitem[Kim02]{kim-sym4}
H.H. Kim.
\newblock {Functoriality for the exterior square of $\GL_4$ and the symmetric
  fourth of $\GL_2$}.
\newblock {\em J.A.M.S}, 16(1):139--183, 2002.

\bibitem[KS02]{kim2002csp}
H.~Kim and F.~Shahidi.
\newblock {Cuspidality of symmetric powers with applications}.
\newblock {\em Duke Math. J}, 112(1):177--197, 2002.

\bibitem[MVW84]{mvw}
C.~R. Matthews, L.~N. Vaserstein, and B.~Weisfeiler.
\newblock {Congruence Properties of Zariski-Dense Subgroups I}.
\newblock {\em Proc. Lon. Math. Soc.}, 48:514--532, 1984.

\bibitem[Nor87]{nori1987sgn}
M.V. Nori.
\newblock {On subgroups of $\GL_n(\F_p)$}.
\newblock {\em Invent. Math}, 88(2):257--275, 1987.

\bibitem[Tay03]{taylor2003rcf}
R.~Taylor.
\newblock {Remarks on a conjecture of Fontaine and Mazur}.
\newblock {\em Journal of the Institute of Mathematics of Jussieu},
  1(1):125--143, 2003.

\bibitem[Tay06]{taylor2006asa}
R.~Taylor.
\newblock {Automorphy for some $l$-adic lifts of automorphic mod $l$ Galois
  representations. II}.
\newblock {\em preprint}, 2006.

\end{thebibliography}

\end{document}